\documentclass[journal]{IEEEtran}

\usepackage[T1]{fontenc}
\usepackage{times}

\usepackage{amsmath,amssymb,amsthm,amsfonts}
\usepackage{graphicx}
\usepackage{booktabs}
\usepackage{algorithm}
\usepackage{algpseudocode}
\usepackage{xcolor}
\usepackage{caption}
\usepackage{subcaption}
\usepackage{cite}
\usepackage{url}
\usepackage[hidelinks]{hyperref}

\usepackage{parskip}
\usepackage{etoolbox}
\AtBeginEnvironment{thebibliography}{\setlength{\parskip}{0pt plus 0pt minus 0pt}}

\newtheorem{theorem}{Theorem}
\newtheorem{proposition}[theorem]{Proposition}
\newtheorem{remark}[theorem]{Remark}

\DeclareMathOperator{\diag}{diag}

\DeclareMathOperator*{\argmin}{arg\,min}

\newcommand{\R}{\mathbb{R}}
\newcommand{\sgm}{\sigma}
\newcommand{\proj}{\mathcal{P}}
\newcommand{\spec}{\mathcal{S}}
\newcommand{\Kset}{\mathcal{K}}
\newcommand{\Cset}{\mathcal{C}}
\newcommand{\Ellipsoid}{\mathcal{E}}

\begin{document}

\title{On feasibility problems with spectral constraints}

\author{Shravan~Mohan\\
17-004, Mantri Residency, Bangalore
\thanks{This is independent research and is not associated with the
author's employer (e-mail: shravan.rammohan@gmail.com).}}

\markboth{Feasibility Problems with Polyhedral Spectral Constraints}%
{Mohan: Alternating Projections for Feasibility Problems with Polyhedral Spectral Constraints}

\maketitle

\begin{abstract}
\textbf{We study matrix feasibility problems ``find
$X \in \mathcal{K} \cap \mathcal{C}$'' where $\mathcal{K}$ is a closed
convex matrix set and $\mathcal{C} = \{X : \sigma(X) \in \mathcal{S}\}$
is defined by a convex constraint $\mathcal{S}$ on the ordered singular
values. Using the classical spectral transfer identity, projection onto
$\mathcal{C}$ reduces to an SVD plus a small quadratic program. For
seven natural polyhedral $\mathcal{S}$ we embed this projector in a
plain alternating projection (AP) loop. We experiment with two concrete
families of $\mathcal{K}$ -- linear constraints (an affine subspace
intersected with an entrywise box) and ellipsoidal constraints (a
non-centered anisotropic Frobenius ellipsoid) -- although the method
applies equally to more general convex constraints. The experiments
expose three regimes: rapid feasibility, slow tail convergence, or
informative infeasibility plateaus.}
\end{abstract}

\begin{IEEEkeywords}
\textbf{Alternating projections, convex feasibility, spectral constraints,
nuclear norm, condition number.}
\end{IEEEkeywords}

\IEEEpeerreviewmaketitle

\section{Introduction}\label{sec:intro}

\IEEEPARstart{M}{any} design and estimation tasks reduce to finding a
matrix that simultaneously meets \emph{structural} requirements on its
entries (or on linear functionals of them) and \emph{spectral}
requirements on its singular values. We formalise this as the abstract
feasibility problem
\begin{equation}\label{eq:feas}
\text{find } X \in \Kset \cap \Cset,
\qquad
\Cset = \{ X \in \R^{m \times n} : \sgm(X) \in \spec \},
\end{equation}
where $\Kset \subset \R^{m \times n}$ is a closed convex set encoding the
structural constraints, and $\Cset$ collects all matrices whose ordered
singular value vector
$\sgm(X) = (\sgm_1, \dots, \sgm_r)$, with $r = \min(m,n)$ and
$\sgm_1 \ge \sgm_2 \ge \cdots \ge \sgm_r \ge 0$, lies in a convex
spectral set $\spec$ defined on the ordered nonnegative cone (polyhedral
in our concrete instances). When a proximity objective
$\min_X \|X - A\|_F$ is added, the same template asks for the
\emph{nearest} structurally- and spectrally-admissible matrix to a
given $A$.

This single template subsumes a number of standard problems, each
obtained by an explicit choice of $\Kset$ and $\spec$.

\begin{itemize}
\item \textbf{Preconditioning to a target condition number.}
Given a design or system matrix $A$, find the nearest $X$ with
$\kappa(X) = \sgm_1(X)/\sgm_r(X) \le \kappa_{\max}$. Here
$\spec = \{ s : s_1 \le \kappa_{\max}\, s_r \}$ is a single linear
inequality on the singular values, while $\Kset$ enforces proximity to
$A$ together with any required entry pattern. The condition number is
exactly the factor by which relative errors can be amplified in
downstream linear solves and least-squares fits, so capping it improves
their numerical accuracy and stability~\cite[Ch.~2--3]{golub2013}
(the standard reference on matrix conditioning and error propagation).

\item \textbf{Weighted nuclear-norm low-rank recovery.}
With linear measurements $\mathcal{A}(X) = b$, take
$\Kset = \{ X : \mathcal{A}(X) = b,\ |X_{ij}| \le \beta \}$ (affine plus
box) and the weighted nuclear-norm constraint
$\spec = \{ s : w^\top s \le \tau \}$, i.e.\
$\sum_{i} w_i\,\sgm_i(X) \le \tau$, for an arbitrary positive weight
vector $w \in \R^r_{>0}$. Unequal weights shrink large and small
singular values differently and outperform the plain nuclear norm
($w = \mathbf{1}$) in low-rank recovery and denoising~\cite{gu2017wnnm}
(which introduces weighted nuclear-norm minimization and shows its
advantage over the plain nuclear norm on image-restoration tasks);
the projection onto $\spec$ remains a convex QP for any $w$.

\item \textbf{Gain shaping for system matrices (control).}
The spectral norm $\sgm_1(X)$ is the worst-case $\ell_2$ gain of the
linear map $x \mapsto Xx$, and $\sgm_r(X)$ its minimum gain.
Singular-value loop shaping confines the gain to a band,
$\spec = \{ s : \ell \le s_i \le u \}$: the upper bound limits
amplification (robustness), while the lower bound $\sgm_1 \ge \ell$
guarantees a minimum gain in every direction, avoiding
near-rank-deficiency and preserving performance (singular-value
loop-shaping for multivariable feedback: the textbook treatment
of~\cite{skogestad2005} and the classical synthesis of~\cite{doyle1981}).
For a discrete-time map $x_{k+1} = Xx_k$, choosing $u < 1$ additionally
certifies a contraction, a sufficient stability condition since
$\rho(X) \le \sgm_1(X)$. With a prescribed sparsity pattern (an affine
$\Kset$) this returns the nearest structurally-admissible matrix with
band-limited gain.

\item \textbf{Spectrum shaping and whitening.}
Bounding the spread $\sgm_1 - \sgm_r \le \delta$ or the successive gaps
$\sgm_i - \sgm_{i+1} \le \delta$ flattens the singular spectrum. The
fully flattened limit is the classical whitening / decorrelation
transform~\cite{kessy2018} (a survey of optimal whitening transforms and
their statistical properties), while a bounded spread yields robust,
well-balanced preconditioners.
\end{itemize}

\textit{A convexity caveat.} In every example above the singular-value
set $\spec$ is convex (a linear inequality, a box, or the simplex on the
ordered nonnegative cone), so the vector projection $\proj_\spec$ -- and
hence the matrix projection $\proj_\Cset$ of Section~\ref{sec:proof} -- is
exact. The induced \emph{matrix} set $\Cset$, however, is convex only for
constraints that upper-bound the largest singular values (the spectral
norm, the nuclear / Ky~Fan norms, and weighted nuclear norms with
non-increasing weights). Constraints that bound the smallest singular
values from below -- the condition number, the gain lower bound, and the
spread / Lipschitz caps -- make $\Cset$ non-convex, because $\sgm_r$ and
the intermediate $\sgm_i$ are not concave functions of $X$. In those
cases the projection identity still holds verbatim and we run AP as a
heuristic. A full derivation of each reduction is given in
Appendix~\ref{app:apps}.

All four are instances of~\eqref{eq:feas}, differing only in the choice
of $\Kset$ and $\spec$. Our contribution is computational. We:
\begin{enumerate}
\item catalogue seven natural polyhedral instances of $\spec$ that cover
the use cases above;
\item implement the exact projector $\proj_\Cset$ as a single SVD plus
one small QP, modeled in the CVXPY convex-optimization
language~\cite{cvxpy} and solved by the OSQP operator-splitting QP
solver~\cite{osqp};
\item drop it into a plain alternating projection loop and run a focused
sweep study on two qualitatively different families of $\Kset$
(polyhedral; ellipsoidal); and
\item report convergence histories and infeasibility witnesses.
\end{enumerate}
Splitting methods such as Douglas--Rachford or ADMM, often used in
the same setting (see~\cite{bauschke} for the underlying projection and
monotone-operator theory), are deliberately omitted here in favour of a
clean exposition of the AP behaviour. The spectral transfer
identity~\eqref{eq:projC} that underpins the projector is
classical~--~it follows from Lewis's convex analysis of unitarily
invariant matrix functions~\cite{lewis1995} and is treated as a proximal
operator in~\cite{beck}; our focus is instead on its behaviour when
embedded in an alternating projection loop across the range of spectral
constraints and convex sets considered here.

\section{Problem Setup}\label{sec:setup}

\subsection{Notation}
For $X \in \R^{m\times n}$ with $r = \min(m,n)$, let
$X = U \diag(\sgm) V^\top$ denote its thin SVD, with the singular
values sorted as
$\sgm_1 \ge \cdots \ge \sgm_r \ge 0$.
We work throughout in the Frobenius inner product
$\langle X, Y\rangle = \mathrm{tr}(X^\top Y)$
with associated norm $\|\cdot\|_F$.
For a closed convex set $\mathcal{S}$, $\proj_{\mathcal{S}}$ denotes
Euclidean projection.

\subsection{Spectral sets and the projection identity}
We say $\spec \subset \R^r$ is \emph{admissible} if it is closed, convex,
contained in the ordered nonnegative cone
$\R^r_{\ge 0,\downarrow} = \{s : s_1 \ge \cdots \ge s_r \ge 0\}$,
and \emph{permutation invariant on the nonnegative orthant} in the sense
that $s \in \spec \iff Ps \in \spec$ for any permutation $P$ acting on
$\R^r_{\ge 0}$ (and projecting back to the ordered cone via sorting).
Let
\begin{equation}\label{eq:Cset}
\Cset \;=\; \{ X \in \R^{m\times n} : \sgm(X) \in \spec \}.
\end{equation}
By construction $\Cset$ is \emph{orthogonally invariant}: if $X \in
\Cset$ then $QXR \in \Cset$ for any orthogonal $Q \in \R^{m\times m}$
and $R \in \R^{n\times n}$.

\subsection{Reducing the matrix projection to a vector projection}
\label{sec:proof}
The Euclidean projection $\proj_\Cset(A)$ is by definition the
minimiser of
\begin{equation}\label{eq:proj_problem}
\min_{X \in \R^{m\times n}} \; \tfrac{1}{2}\|X - A\|_F^2
\quad \text{s.t.}\quad \sgm(X) \in \spec.
\end{equation}
We now reduce \eqref{eq:proj_problem} to a vector problem and exhibit
an optimal $X^\star$ that shares its singular vectors with $A$. The
key tool is the following classical inequality.

\begin{theorem}[von Neumann's trace inequality~\cite{vonNeumann1937}]
\label{thm:vN}
For any $X, A \in \R^{m\times n}$ with singular values
$\sgm(X)$ and $\sgm(A)$ in $\R^r_{\ge 0,\downarrow}$,
\begin{equation}\label{eq:vN}
\langle X, A\rangle \;=\; \mathrm{tr}(X^\top A)
\;\le\; \sum_{i=1}^{r} \sgm_i(X)\,\sgm_i(A).
\end{equation}
Moreover, equality is attained if and only if $X$ and $A$ admit a
\emph{simultaneous SVD}, i.e.\ there exist orthogonal matrices
$U \in \R^{m\times m}$, $V \in \R^{n\times n}$ such that
\begin{equation}\label{eq:sameUV}
X = U \diag(\sgm(X))\,V^\top, \qquad
A = U \diag(\sgm(A))\,V^\top.
\end{equation}
\end{theorem}

\begin{theorem}[Spectral projection identity~\cite{lewis1995,beck}]
\label{thm:proj}
Let $\spec$ be admissible and let $A \in \R^{m\times n}$ have thin SVD
$A = U \diag(s) V^\top$ with $s = \sgm(A) \in \R^r_{\ge 0,\downarrow}$.
Then problem \eqref{eq:proj_problem} attains its minimum at
\begin{equation}\label{eq:projC}
X^\star \;=\; U\,\diag\!\bigl(\proj_\spec(s)\bigr)\,V^\top,
\end{equation}
where $\proj_\spec$ denotes Euclidean projection on $\R^r$. Equivalently,
$\proj_\Cset(A)$ shares the singular vectors $(U, V)$ of $A$, and its
singular values are obtained by projecting $\sgm(A)$ onto $\spec$.
\end{theorem}

\begin{proof}
Expanding the objective of \eqref{eq:proj_problem} gives
\begin{equation}\label{eq:expand}
\tfrac{1}{2}\|X - A\|_F^2 \;=\; \tfrac{1}{2}\|X\|_F^2 + \tfrac{1}{2}\|A\|_F^2 - \langle X, A\rangle.
\end{equation}
The Frobenius norm is unitarily invariant, so
$\|X\|_F^2 = \sum_i \sgm_i(X)^2$ and similarly for $A$. By the
trace inequality \eqref{eq:vN},
\begin{equation}\label{eq:lb}
\langle X, A\rangle \;\le\; \sgm(X)^\top \sgm(A),
\end{equation}
and combining \eqref{eq:expand} and \eqref{eq:lb},
\begin{align}
\tfrac{1}{2}\|X - A\|_F^2
&\;\ge\;
\tfrac{1}{2}\|\sgm(X)\|_2^2 + \tfrac{1}{2}\|\sgm(A)\|_2^2 - \sgm(X)^\top \sgm(A) \notag \\
&\;=\;
\tfrac{1}{2}\|\sgm(X) - \sgm(A)\|_2^2.
\label{eq:lb2}
\end{align}
For any feasible $X \in \Cset$, the right-hand side of \eqref{eq:lb2}
depends only on $\sgm(X) \in \spec$. Hence the optimal value of
\eqref{eq:proj_problem} is at least
\begin{equation}\label{eq:vec_lower}
\min_{t \in \spec}\;\tfrac{1}{2}\|t - \sgm(A)\|_2^2
\;=\;\tfrac{1}{2}\|\proj_\spec(s) - s\|_2^2.
\end{equation}

We now show that this lower bound is attained by the choice
\eqref{eq:projC}. Let $t^\star = \proj_\spec(s) \in \spec$, and set
$X^\star = U \diag(t^\star) V^\top$ with the same orthogonal matrices
$U, V$ as in the SVD of $A$. Then $\sgm(X^\star) = t^\star \in \spec$,
so $X^\star \in \Cset$. Moreover $X^\star$ and $A$ trivially admit a
simultaneous SVD \eqref{eq:sameUV}, so equality holds in
\eqref{eq:vN}, hence in \eqref{eq:lb}, hence in \eqref{eq:lb2}. The
objective at $X^\star$ therefore equals
$\tfrac{1}{2}\|t^\star - s\|_2^2$, which matches \eqref{eq:vec_lower}.
\end{proof}

\begin{remark}[Why $U$ and $V$ are inherited from $A$]
The proof shows that aligning the singular vectors of $X$ with those of
$A$ is both \emph{necessary} (to realise equality in von Neumann's
inequality) and \emph{sufficient} (combined with any $t \in \spec$) to
attain the lower bound \eqref{eq:vec_lower}. In particular, any other
choice of orthogonal matrices yields a strictly suboptimal $X$.
\end{remark}

\begin{remark}[The vector subproblem]
Theorem~\ref{thm:proj} reduces a projection in $\R^{mn}$ to one in
$\R^r$. For all seven instances of $\spec$ in
Section~\ref{sec:polyhedral}, the vector subproblem is a strictly
convex QP on $r$ variables, solved here through CVXPY/OSQP. The total
cost of $\proj_\Cset(A)$ is therefore one thin SVD ($O(mn\,\min(m,n))$)
plus one small QP. \newline 
\end{remark}

\paragraph*{Method summary}
Given $A \in \R^{m\times n}$, computing $\proj_\Cset(A)$ proceeds as
follows.
\begin{enumerate}
\item Compute the thin SVD $A = U\,\diag(s)\,V^\top$ with $s$ sorted in
      descending order.
\item Solve the vector projection
      $t^\star = \argmin_{t \in \spec}\|t - s\|_2^2$.
\item Return $X^\star = U\,\diag(t^\star)\,V^\top$.
\end{enumerate}
By Theorem~\ref{thm:proj}, $X^\star$ is the unique\footnote{Uniqueness
holds whenever the singular values of $A$ are distinct; in degenerate
cases the projection is unique only in the ``spectral'' sense, i.e.\
$\sgm(X^\star)$ is unique.} minimiser of \eqref{eq:proj_problem}.

\section{Polyhedral Spectral Constraints}\label{sec:polyhedral}

We focus on the case where $\spec$ is the intersection of the ordered
nonnegative cone with a polyhedron $\{s : Ls \le b\}$; computing
$\proj_\spec$ is then a strictly convex QP in $r$ variables, solved here
through CVXPY using the OSQP solver~\cite{osqp}. Throughout, we impose
the ordering and nonnegativity constraints $s_1 \ge \cdots \ge s_r \ge 0$,
and the seven cases below differ only in the additional inequalities
$Ls \le b$.

The restriction to polyhedral $\spec$ is purely for concreteness. The
spectral projection identity of Theorem~\ref{thm:proj} uses nothing
about $\spec$ beyond the existence of its projector $\proj_\spec$, so
$\spec$ may be any closed convex set on the ordered singular values --- a
norm ball, a quadratic or second-order-cone region, or any other convex
constraint --- and the method is unchanged: one SVD followed by the
(now possibly conic) vector projection. The polyhedral instances merely
make $\proj_\spec$ a quadratic program.

\begin{enumerate}
\item \textbf{Condition number.}
$\sgm_1 \le \kappa \, \sgm_r$, equivalent to a single linear inequality
$\sgm_1 - \kappa \sgm_r \le 0$.

\item \textbf{Spectral spread.}
$\sgm_1 - \sgm_r \le \delta$, again a single linear constraint.

\item \textbf{Singular value box.}
$\ell_i \le \sgm_i \le u_i$ for $i=1,\dots,r$. With $\ell = 0$ this
reduces to clipping; in general the ordering constraint makes it a
non-trivial QP.

\item \textbf{Weighted cap.}
$w^\top \sgm \le c$ for a fixed positive weight vector
$w \in \R^r_{>0}$. With $w = \mathbf{1}$ this is the nuclear-norm
constraint $\|X\|_* \le c$.

\item \textbf{Arbitrary polyhedron.}
$L \sgm \le b$ with $L \in \R^{p\times r}$.

\item \textbf{Lipschitz spectrum.}
$|\sgm_i - \sgm_{i+1}| \le \delta$ for $i=1,\dots,r-1$. Combined with
the ordering constraint this collapses to
$\sgm_i - \sgm_{i+1} \le \delta$.

\item \textbf{Spectral simplex.}
$\sum_i \sgm_i = z,\ \sgm_i \ge 0$. Admits a $O(r\log r)$ closed-form
projection~\cite{simplexproj} (a sort-and-threshold algorithm for
Euclidean projection onto the probability simplex).
\end{enumerate}

\section{Algorithm: Alternating Projections}\label{sec:ap}

Recasting the feasibility problem~\eqref{eq:feas} as a pair of
alternating projections is precisely what makes the approach general:
the two sets $\Kset$ and $\Cset$ enter only through their projectors, so
any pair of sets equipped with a projection oracle can be combined
without altering the algorithm. All of the spectral structure is
confined to the single operator $\proj_\Cset$, and all of the structural
structure to $\proj_\Kset$.

Given $\proj_\Kset$ and $\proj_\Cset$, the von Neumann / Bregman
alternating projection (AP) iteration is
\begin{equation}\label{eq:AP}
X_k = \proj_\Cset(A_k),
\qquad
A_{k+1} = \proj_\Kset(X_k),
\end{equation}
starting from any $A_0$, typically $A_0 = A_{\mathrm{target}}$ for a
nearest-point flavoured problem. We use the early stop rule
\begin{equation}\label{eq:stop}
\text{stop at iter } k \iff \mathrm{metric}(A_k) \le \mathrm{target}\,(1 + \epsilon),
\end{equation}
with $\epsilon = 10^{-6}$, where $\mathrm{metric}$ is the scalar
spectral functional being constrained (e.g.\ $\kappa$, spread,
weighted sum). We also cap the iteration count at $100$.

\begin{algorithm}[t]
\caption{Alternating projection for $\Kset \cap \Cset$.}\label{alg:AP}
\begin{algorithmic}[1]
\Require $A_0 \in \R^{m\times n}$, projectors
$\proj_\Kset$, $\proj_\Cset$, max-iter $T$, tolerance $\epsilon$,
metric $\phi$ and target $\tau$.
\For{$k = 0, 1, \dots, T-1$}
  \State $X_k \gets \proj_\Cset(A_k)$
  \State $A_{k+1} \gets \proj_\Kset(X_k)$
  \If{$\phi(A_{k+1}) \le \tau(1+\epsilon)$}
    \State \Return $A_{k+1}$ \Comment{feasible early stop}
  \EndIf
\EndFor
\State \Return $A_T$ \Comment{may be infeasible}
\end{algorithmic}
\end{algorithm}

\begin{proposition}[Monotone gap]\label{prop:gap}
For any closed convex sets $\Kset, \Cset$ and any starting $A_0$,
\[
\|A_{k+1} - X_{k+1}\| \;\le\; \|A_{k+1} - X_k\| \;\le\; \|A_k - X_k\|,
\]
so the gap $\|A_k - X_k\|$ is non-increasing.
\end{proposition}

\begin{proof}
By construction $X_k = \proj_\Cset(A_k) \in \Cset$ and
$A_{k+1} = \proj_\Kset(X_k) \in \Kset$; in particular $A_k \in \Kset$ for
every $k \ge 1$. The Euclidean projection onto a closed convex set
returns the nearest point of that set, hence it cannot be farther from
the projected point than any other member of the set. Applying this to
$X_{k+1} = \proj_\Cset(A_{k+1})$ with the competitor $X_k \in \Cset$,
\begin{equation}\label{eq:gap1}
\|A_{k+1} - X_{k+1}\| \;\le\; \|A_{k+1} - X_k\|,
\end{equation}
and applying it to $A_{k+1} = \proj_\Kset(X_k)$ with the competitor
$A_k \in \Kset$,
\begin{equation}\label{eq:gap2}
\|A_{k+1} - X_k\| \;=\; \|X_k - A_{k+1}\| \;\le\; \|X_k - A_k\|.
\end{equation}
Chaining \eqref{eq:gap1} and \eqref{eq:gap2} yields the stated
inequalities, so the gap $\|A_k - X_k\|$ is non-increasing.
\end{proof}

We caution that this does \emph{not}
imply convergence to a feasible point when $\Kset \cap \Cset = \emptyset$,
nor (in general) when the intersection is nonempty: convergence is
established under additional regularity (e.g.\ transversality,
Fej\'er monotonicity arguments). We therefore treat AP as a heuristic.

Even when the iteration does not reach a strictly feasible point, the
returned iterate remains a useful surrogate. By
Proposition~\ref{prop:gap}, the gap $\|A_k - X_k\|$ decreases
monotonically, so the final pair $(A_k, X_k)$ is the most mutually
consistent one the iteration produced: $A_k$ exactly satisfies the
structural constraints ($A_k \in \Kset$), $X_k$ exactly satisfies the
spectral constraints ($X_k \in \Cset$), and the two differ by the
smallest residual attained. In the infeasible regime this residual
quantifies how far the two requirements can be reconciled, and either
iterate can be reported as a near-feasible solution that violates only
one set, and only by $\|A_k - X_k\|$. The infeasibility plateaus in
Section~\ref{sec:results} are exactly such surrogates, and their plateau
values carry quantitative meaning (Observation~(O3)).

\section{Experimental Setup}\label{sec:setup-exp}

All experiments use a fixed problem of dimensions $m=20$, $n=15$,
$r=15$ matrices, with a planted feasible interior point
$X_\mathrm{feas}$ and an $A_\mathrm{target} = X_\mathrm{feas} +
0.5 \xi$ for $\xi$ Gaussian. For each of the seven spectral
constraints we sweep the parameter that defines $\spec$ (e.g.\
$\kappa$, $\delta$, $c$, $u$, $\tau$, or $z$) over an ordered schedule
that tightens the constraint -- deliberately extending it past the
point of feasibility -- and re-run Algorithm~\ref{alg:AP} from
$A_0 = A_\mathrm{target}$ for each convex set $\Kset$.

\subsection{Choice of $\Kset$}
We compare two qualitatively different convex matrix sets:

\paragraph{Affine + box}
$\Kset = \Kset_{\mathrm{aff}} =
\{ X : A_\mathrm{lin} X = B_\mathrm{lin},\ -1 \le X_{ij} \le 1 \}$,
with $A_\mathrm{lin} \in \R^{5\times m}$ Gaussian and $B_\mathrm{lin}$
chosen so that $X_\mathrm{feas} \in \Kset$.

\paragraph{Ellipsoid}
$\Kset = \Ellipsoid =
\bigl\{ X : \sum_{ij} w_{ij}^2 (X_{ij} - X_c^{ij})^2 \le r^2 \bigr\}$,
with random anisotropic per-entry weights $w_{ij} \in [0.5,2]$, a
random non-zero center $X_c$ (in our seed,
$\|X_c\|_F \approx 5.17$), and radius chosen with a 5\% slack around
$X_\mathrm{feas}$.

\subsection{Constraint metrics}
For each of the seven spectral cases the early-stop metric $\phi(A)$ is
the scalar quantity bounded by $\spec$:
\[
\begin{array}{ll}
\kappa & \phi(A) = \sgm_1(A)/\sgm_r(A) \\
\text{spread} & \phi(A) = \sgm_1(A) - \sgm_r(A) \\
\text{box} & \phi(A) = \max(\sgm_1(A)-u,\ l-\sgm_r(A),\ 0) \\
\text{weighted} & \phi(A) = w^\top \sgm(A) \\
\text{polyhedral} & \phi(A) = \max_j (L\sgm(A))_j / b^0_j \\
\text{Lipschitz} & \phi(A) = \max_i (\sgm_i(A) - \sgm_{i+1}(A)) \\
\text{simplex} & \phi(A) = \textstyle\sum_i \sgm_i(A)
\end{array}
\]
Feasibility is declared when $\phi(A_k) \le \tau$ for the inequality
constraints, or $|\phi(A_k) - z| \le 10^{-2}z$ for the simplex equality.
The box case is genuinely two-sided: we fix a positive lower bound
$l=0.5$ on every singular value and tighten the upper bound $u$; the
metric $\phi$ is then the box violation (zero iff $l \le \sgm_i \le u$
for all $i$, so $\tau=0$), capturing both bounds at once.
For the polyhedral sweep we fix a random nonnegative $L$, set the
baseline $b^0 = L\,\sgm(A_\mathrm{target})$, and tighten the right-hand
side as $b = \tau b^0$; the metric is then the worst normalised row.
Targets are integer where naturally so, and the weighted cap $c$ is
rounded to the nearest integer.

\section{Results}\label{sec:results}

\subsection{Affine + box $\Kset$ vs ellipsoidal $\Ellipsoid$}

We run all seven spectral constraints on both convex sets, sweeping each
spectral parameter from loose to tight, deliberately into the infeasible
regime. Figure~\ref{fig:affine} collects the seven sweeps for the affine
+ box $\Kset$ and Fig.~\ref{fig:ellip} the seven for the ellipsoidal
$\Ellipsoid$. In every panel feasible runs are drawn as solid lines
(annotated with the early-stop iteration) and infeasible runs as dashed
lines whose legend reports the plateau value reached and its gap to the
target -- i.e.\ exactly how close AP got to the constraint.

For the affine + box $\Kset$ (Fig.~\ref{fig:affine}), the
condition-number, spread, and Lipschitz sweeps reach every target within
$100$ iterations, with the iteration count growing as the constraint
tightens. The remaining four cases expose informative
\emph{infeasibility plateaus}. In the weighted-cap sweep (arbitrary
positive weights $w_i$ printed below the panel), integer caps down to
$c=14$ are met, while $c=8$ plateaus at $w^\top\sgm \approx 12.3$ --- an
achievable value in $\Kset$ that estimates, and upper-bounds, the best
attainable $\min_{X\in\Kset} w^\top\sgm(X)$. The two-sided
box sweep (fixed lower bound $l=0.5$, tightening $u$) is feasible for
$u \ge 2$, but for every $u \le 1.5$ the violation plateaus at a residual
of $\sgm_1 - u \approx 0.12$ at $u=1.5$; AP cannot push $\sgm_1$
below $\approx 1.62$ in this $\Kset$ (the lower bound $l$ is met
throughout). The
polyhedral sweep becomes infeasible only at the tightest scale
($\tau=0.2$, reaching $0.25$), and the simplex equality becomes
unreachable at the smallest target sum ($z\approx 6.2$, plateauing at
$\approx 6.3$).

For the ellipsoidal $\Ellipsoid$ (Fig.~\ref{fig:ellip}), AP converges
far faster -- often in a single iteration -- and \emph{every} target,
across all seven constraints and all parameter values, is met. We
attribute this to the smooth, full-dimensional geometry of
$\Ellipsoid$: there are no flat faces or corners for the iterates to
wedge into, and the ellipsoid intersects essentially any reasonable
spectral set $\Cset$. The polyhedral affine $\Kset$, by contrast, can
force the iterate into a low-dimensional slice where a tightened $\Cset$
is unreachable -- and where the plateau then quantifies the residual
infeasibility.

\begin{figure*}[!t]
\centering
\begin{subfigure}[t]{0.32\textwidth}
  \includegraphics[width=\linewidth]{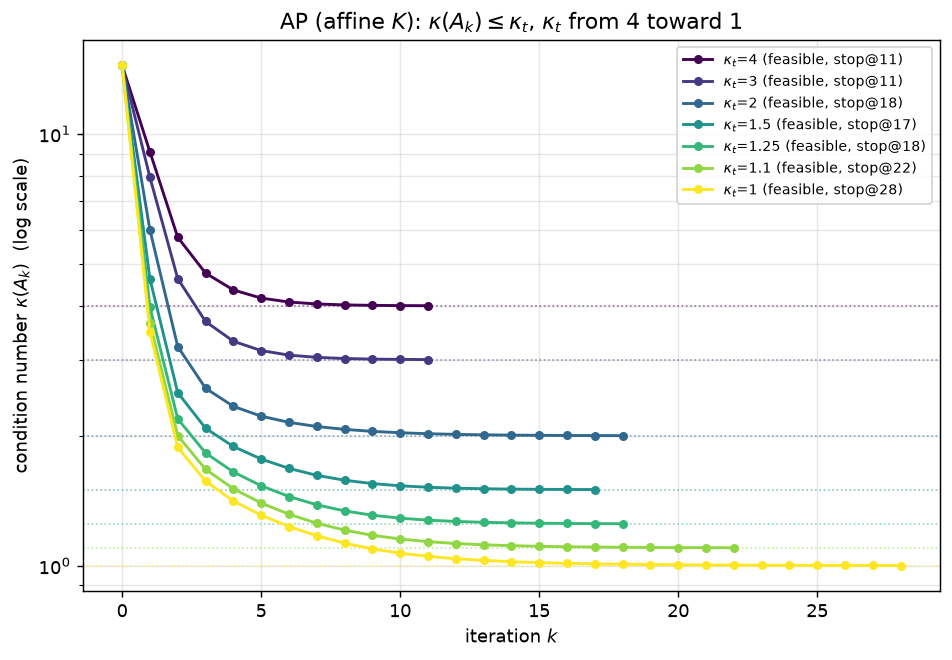}
  \caption{condition number $\kappa$}
\end{subfigure}\hfill
\begin{subfigure}[t]{0.32\textwidth}
  \includegraphics[width=\linewidth]{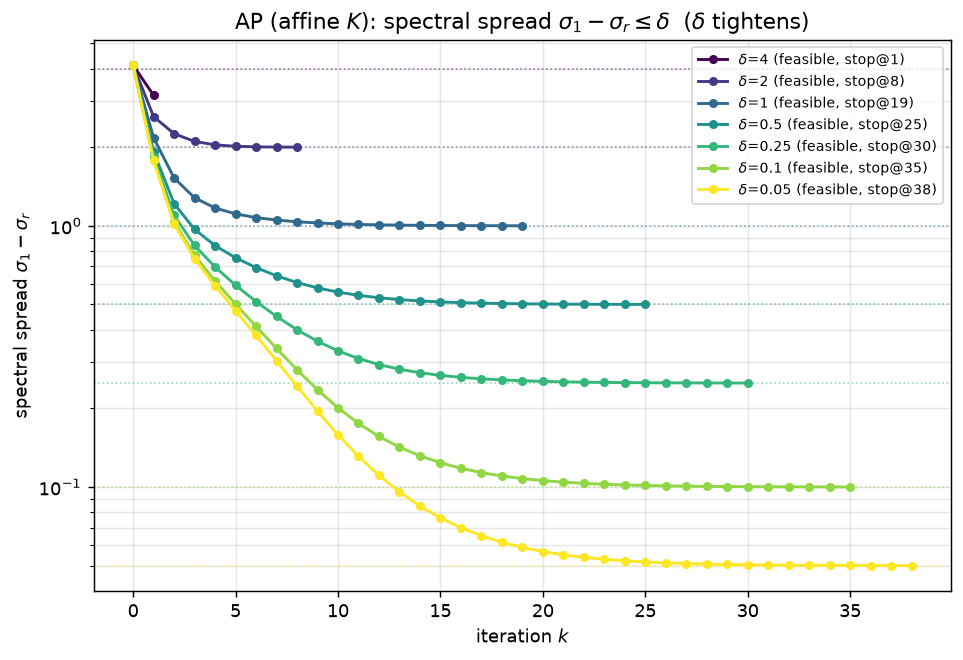}
  \caption{spectral spread}
\end{subfigure}\hfill
\begin{subfigure}[t]{0.32\textwidth}
  \includegraphics[width=\linewidth]{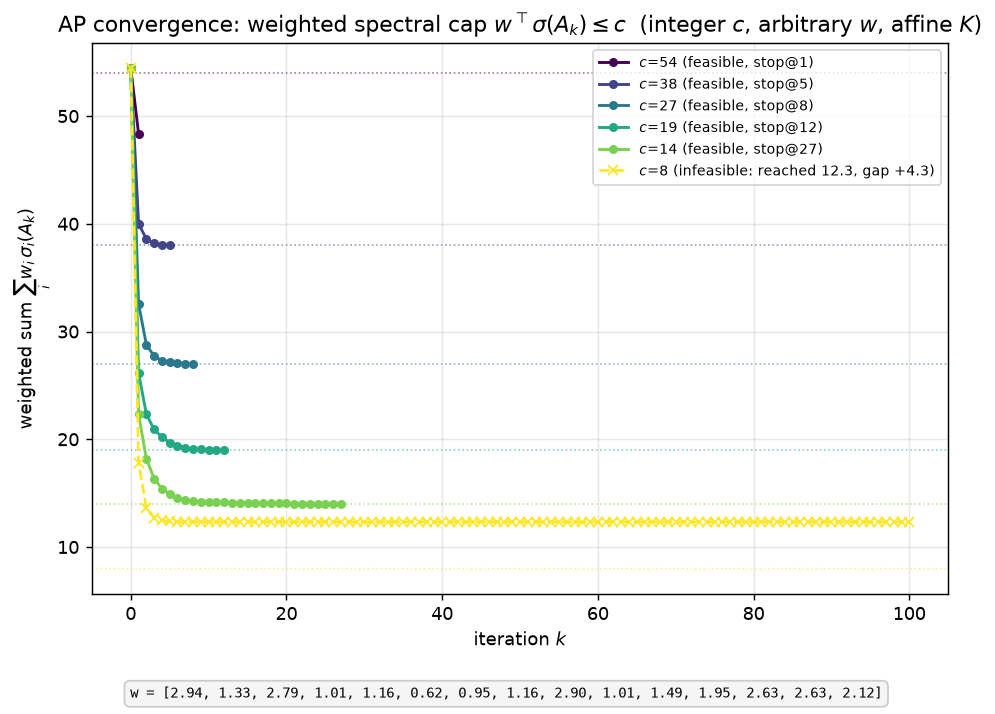}
  \caption{weighted cap}
\end{subfigure}

\vspace{0.5em}

\begin{subfigure}[t]{0.32\textwidth}
  \includegraphics[width=\linewidth]{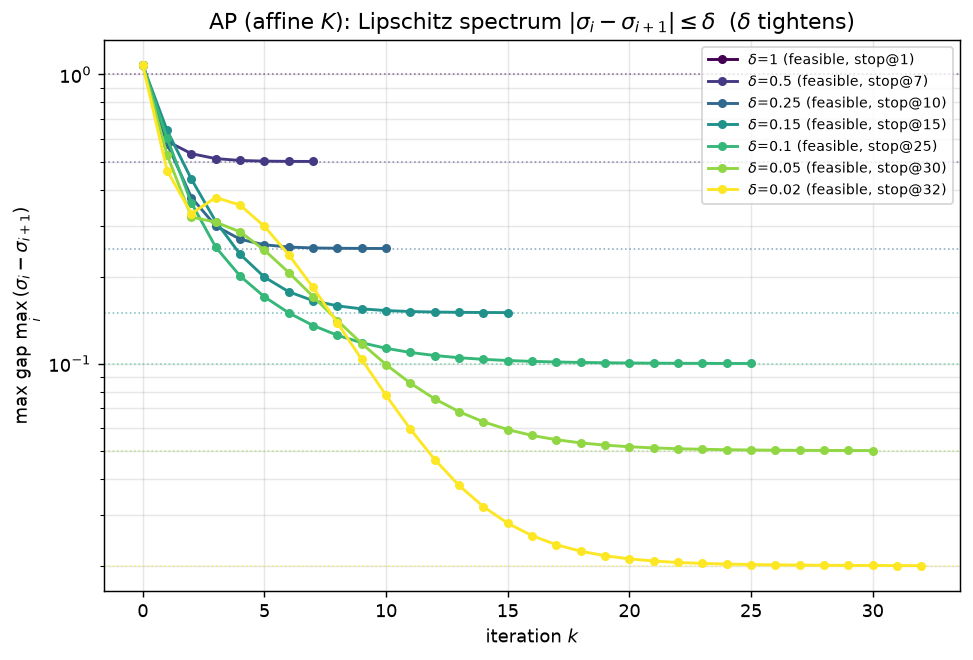}
  \caption{Lipschitz spectrum}
\end{subfigure}\hfill
\begin{subfigure}[t]{0.32\textwidth}
  \includegraphics[width=\linewidth]{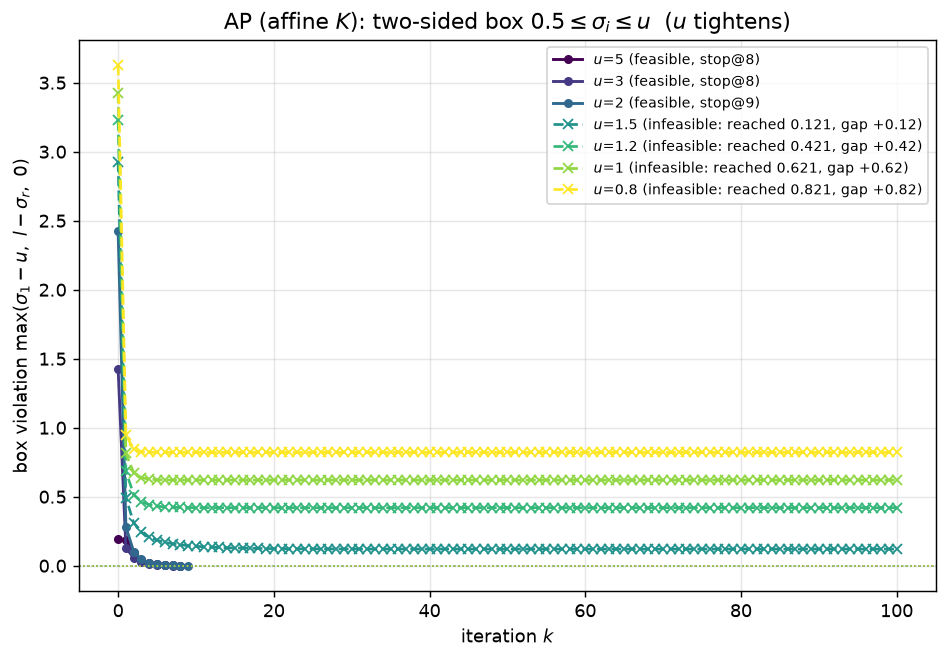}
  \caption{two-sided box}
\end{subfigure}\hfill
\begin{subfigure}[t]{0.32\textwidth}
  \includegraphics[width=\linewidth]{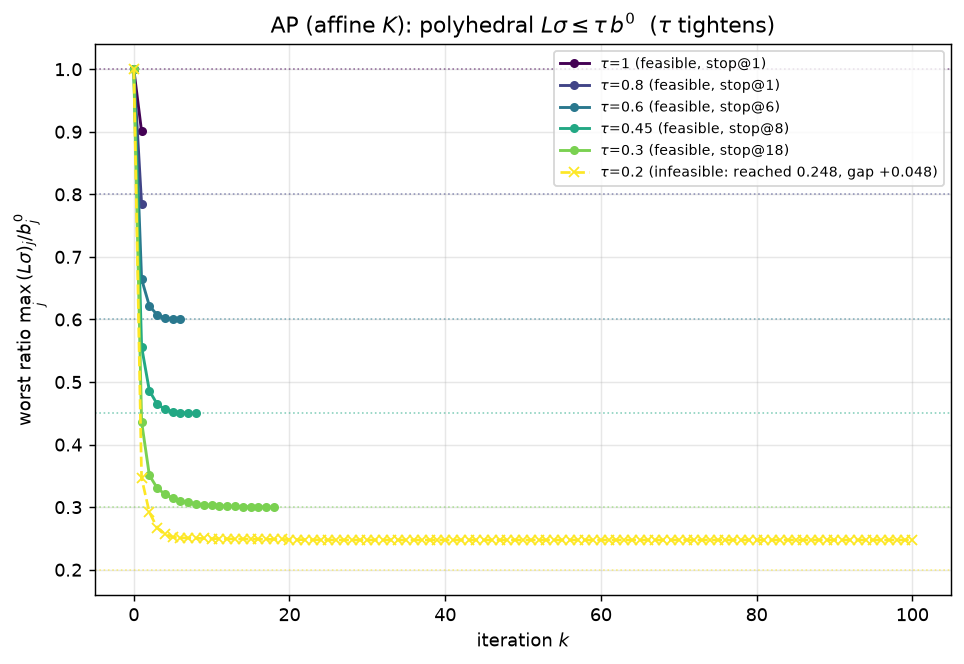}
  \caption{arbitrary polyhedron}
\end{subfigure}

\vspace{0.5em}

\begin{subfigure}[t]{0.32\textwidth}
  \includegraphics[width=\linewidth]{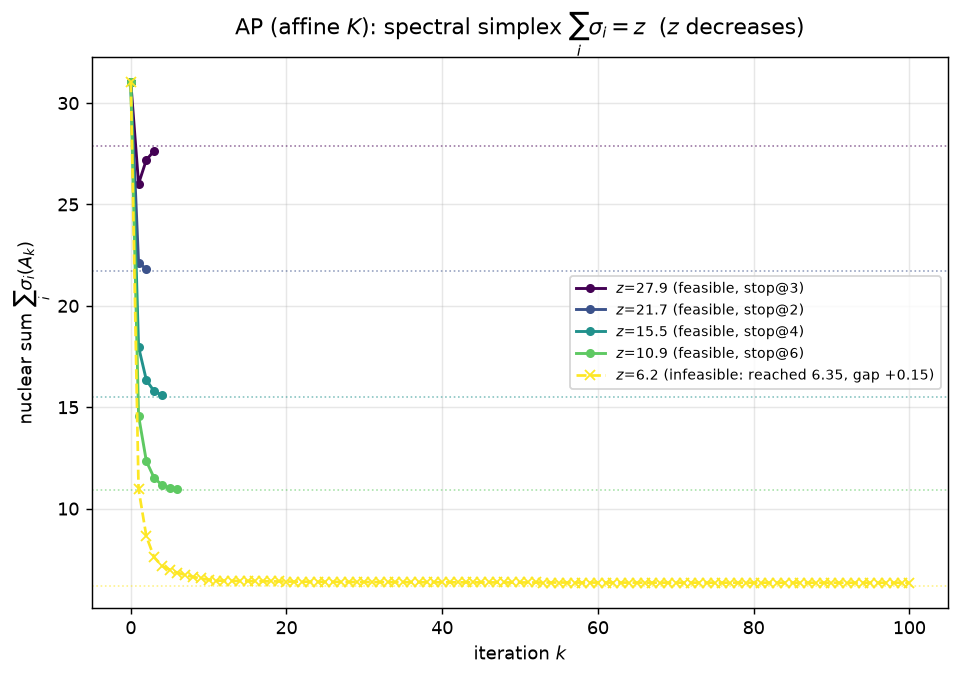}
  \caption{spectral simplex}
\end{subfigure}

\caption{AP convergence for all seven spectral constraints with the
\textbf{affine + box} $\Kset$. Solid lines: feasible runs (legend gives
the early-stop iteration). Dashed lines: infeasible runs (legend gives
the plateau value reached and its gap to target). Dotted horizontals
mark per-curve targets. Infeasibility plateaus appear in the
weighted-cap, two-sided box, polyhedral, and simplex sweeps.}
\label{fig:affine}
\end{figure*}

\begin{figure*}[!t]
\centering
\begin{subfigure}[t]{0.32\textwidth}
  \includegraphics[width=\linewidth]{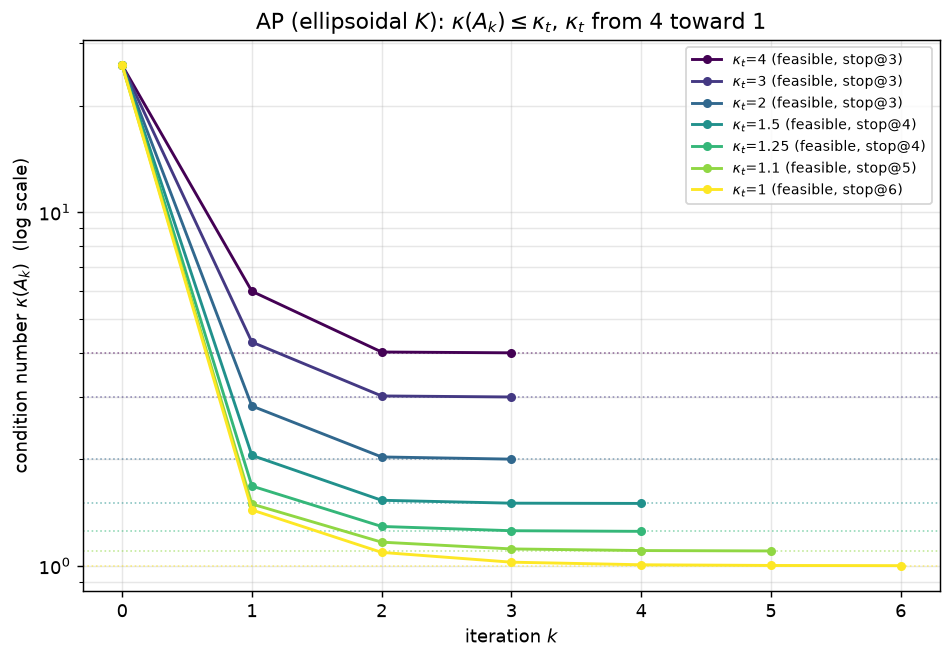}
  \caption{condition number $\kappa$}
\end{subfigure}\hfill
\begin{subfigure}[t]{0.32\textwidth}
  \includegraphics[width=\linewidth]{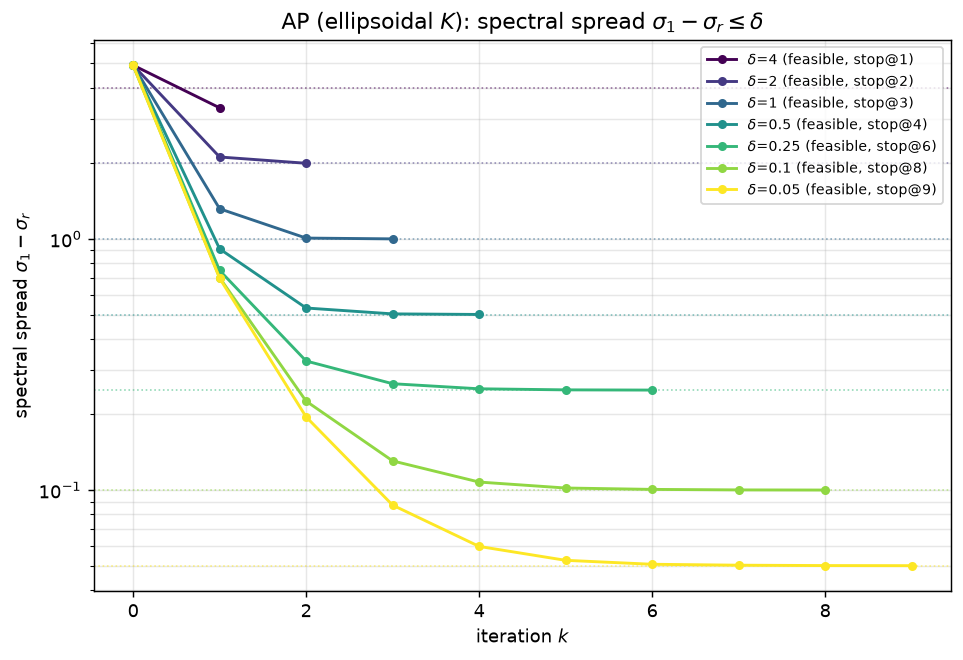}
  \caption{spectral spread}
\end{subfigure}\hfill
\begin{subfigure}[t]{0.32\textwidth}
  \includegraphics[width=\linewidth]{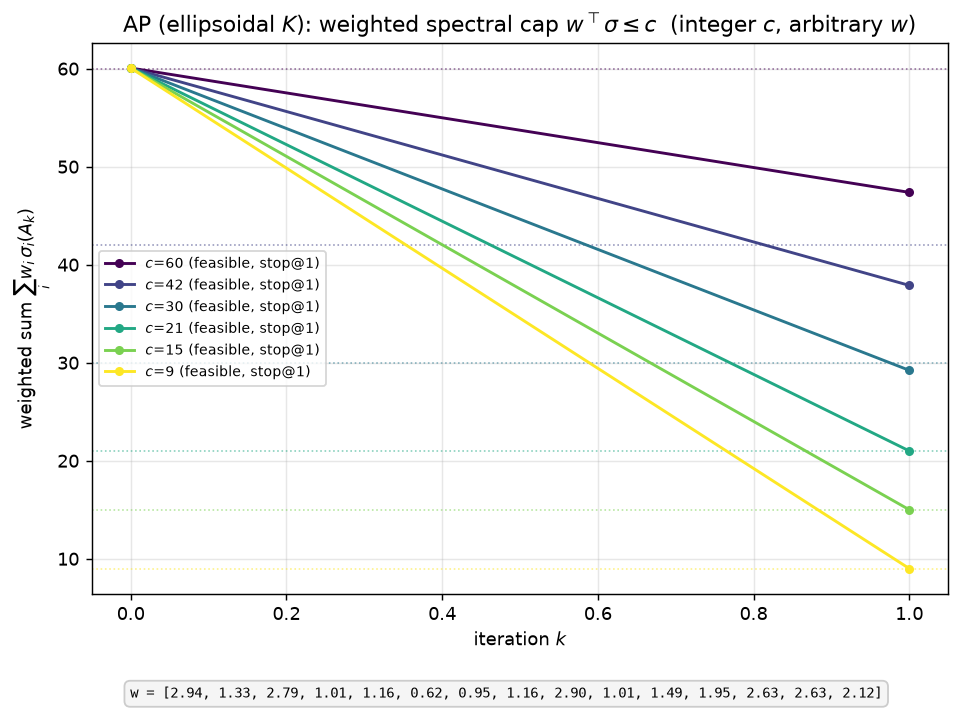}
  \caption{weighted cap}
\end{subfigure}

\vspace{0.5em}

\begin{subfigure}[t]{0.32\textwidth}
  \includegraphics[width=\linewidth]{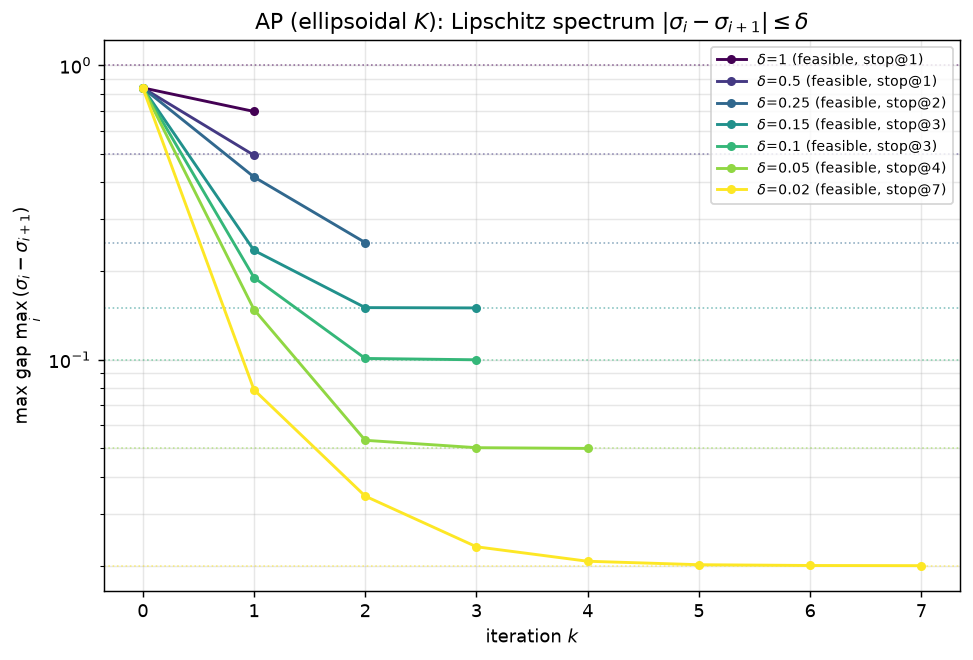}
  \caption{Lipschitz spectrum}
\end{subfigure}\hfill
\begin{subfigure}[t]{0.32\textwidth}
  \includegraphics[width=\linewidth]{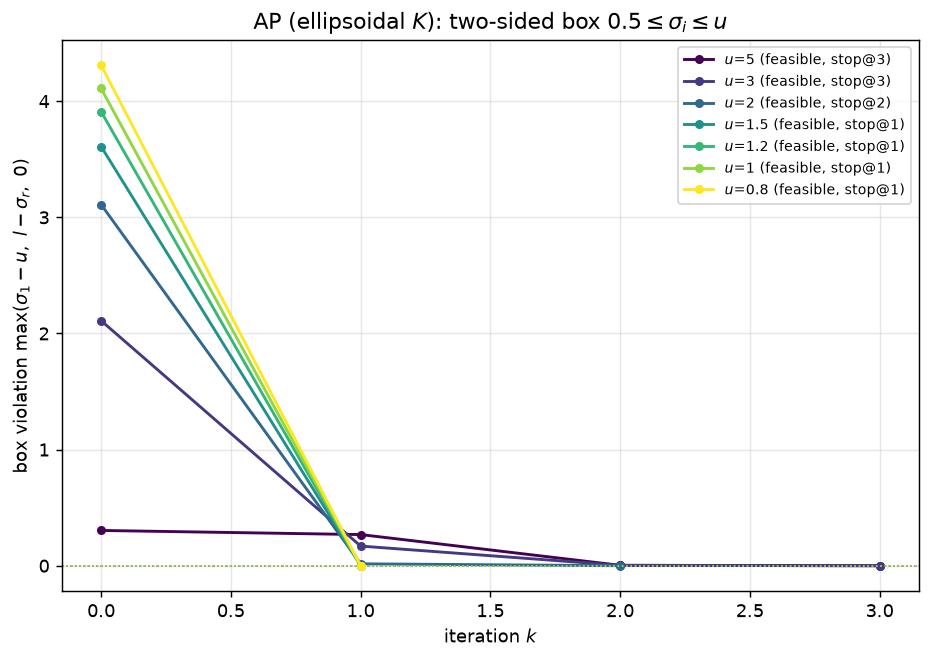}
  \caption{two-sided box}
\end{subfigure}\hfill
\begin{subfigure}[t]{0.32\textwidth}
  \includegraphics[width=\linewidth]{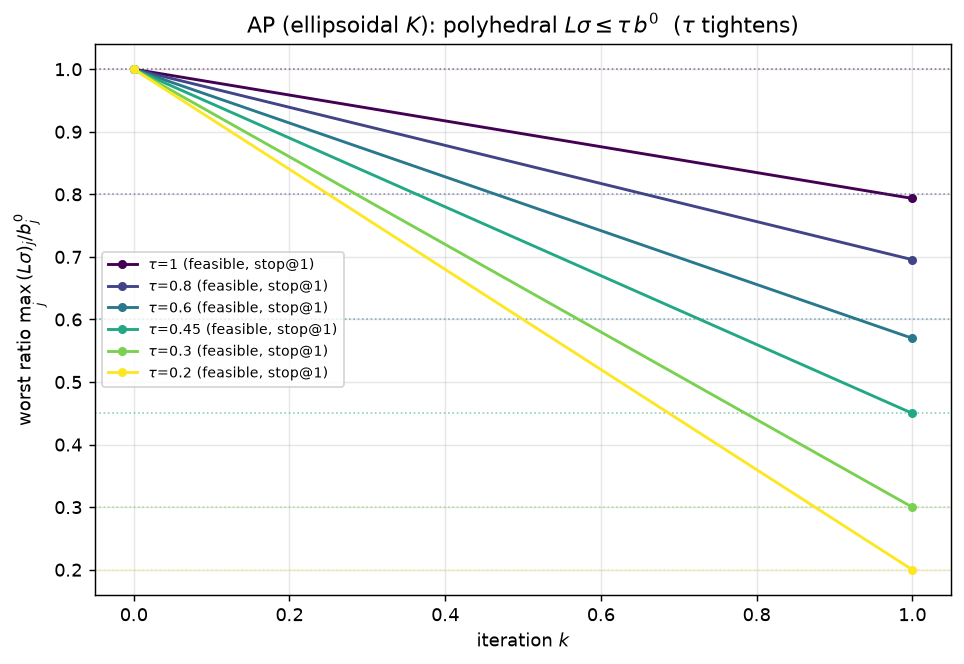}
  \caption{arbitrary polyhedron}
\end{subfigure}

\vspace{0.5em}

\begin{subfigure}[t]{0.32\textwidth}
  \includegraphics[width=\linewidth]{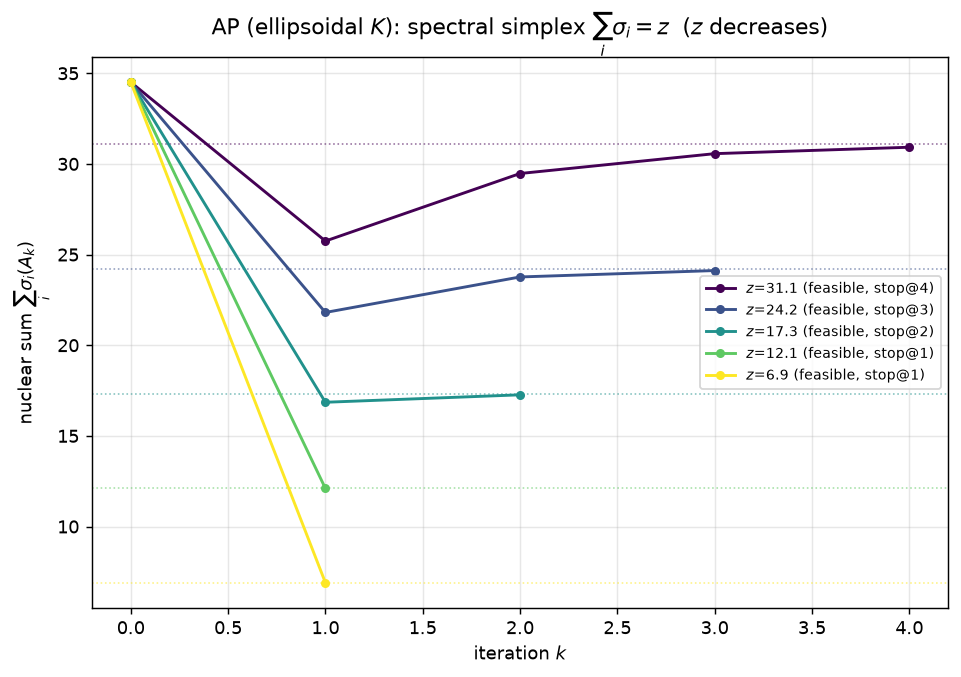}
  \caption{spectral simplex}
\end{subfigure}

\caption{AP convergence for all seven spectral constraints with the
\textbf{ellipsoidal} (non-centered, anisotropic Frobenius) $\Ellipsoid$.
Same plotting convention as Fig.~\ref{fig:affine}. Here every target is
reached, typically within a handful of iterations, and no infeasibility
plateaus arise.}
\label{fig:ellip}
\end{figure*}

\subsection{Summary table}
Table~\ref{tab:summary} consolidates the per-sweep iteration counts and
infeasibility behaviour. Iteration ranges span the loosest to tightest
\emph{feasible} target; an \emph{infeas} entry reports the tightest
parameter that plateaued and the value it reached.

\begin{table}[t]
\caption{AP behaviour for each spectral constraint on the two convex
sets. Range = iterations from loosest to tightest feasible target;
\emph{infeas} = plateaued within 100 iterations (with reached value).}
\label{tab:summary}
\centering
\footnotesize
\setlength{\tabcolsep}{4pt}
\begin{tabular}{lcc}
\toprule
Sweep & Affine+box $\Kset$ & Ellipsoid $\Ellipsoid$ \\
\midrule
$\kappa$: $4 \to 1$             & $11$--$28$                      & $3$--$6$ \\
spread: $\delta\ 4 \to 0.05$    & $1$--$38$                       & $1$--$9$ \\
weighted ($c$ int.)             & $1$--$27$; inf@$12.3$           & $1$ \\
Lipschitz: $\delta\ 1 \to 0.02$ & $1$--$32$                       & $1$--$7$ \\
box: $[0.5,u]$, $u\ 5{\to}0.8$  & $8$--$9$; inf@$\sgm_1{=}1.62$   & $1$--$3$ \\
polyhedral: $\tau\ 1 \to 0.2$   & $1$--$18$; inf@$0.25$           & $1$ \\
simplex: $z\ \downarrow$        & $2$--$6$; inf@$6.3$             & $1$--$4$ \\
\bottomrule
\end{tabular}
\\[2pt]
{\footnotesize inf@$v$ = tightest target infeasible; plateaued at $v$.}
\end{table}

\section{Discussion}\label{sec:disc}

The numerical study supports three observations.

\begin{enumerate}
\item[(O1)] \textbf{The spectral projector is cheap.}
Each call to $\proj_\Cset$ is one thin SVD plus a single QP on $r=15$
variables. For all sweeps, AP completes in $\le 100$ iterations,
typically much fewer.

\item[(O2)] \textbf{Geometry of $\Kset$ matters more than the spectral
case.}
The same seven sweeps look qualitatively different on polyhedral
$\Kset$ vs ellipsoidal $\Kset$. Polyhedral $\Kset$ slows convergence
and admits infeasibility plateaus in four of the seven cases;
ellipsoidal $\Kset$ is smooth and hits feasibility almost immediately
across all seven.

\item[(O3)] \textbf{Infeasibility is informative.}
When AP plateaus above the target, the plateau value
$\lim_k \phi(A_k)$ is still a useful number. Since $A_k \in \Kset$
throughout, it is an \emph{achievable} value of the spectral functional
within the structural constraints, hence an upper bound on the best
attainable value $\min_{X\in\Kset}\phi(X)$. Note that AP minimises the
gap $\|A_k - X_k\|$ rather than $\phi$ over $\Kset$, so we report the
plateau as a heuristic estimate of that minimum, not a proven optimum.
\end{enumerate}

\section{Conclusion}\label{sec:conc}

We presented a uniform framework for matrix feasibility with polyhedral
spectral constraints, exploiting the classical spectral transfer
identity~\eqref{eq:projC} to reduce the matrix projection to a small QP
on $r$ singular values. Embedded in a plain alternating projection
loop, the framework solves a range of spectral feasibility problems
exactly (up to early-stop tolerance) on two qualitatively different
families of convex matrix sets. Future work includes
(a) a rigorous local-convergence analysis for AP on polyhedral $\Kset$
intersected with a smooth $\Cset$, and (b) extension to arbitrary
admissible (not necessarily polyhedral) $\spec$.

\appendices

\section{The Four Applications as Instances of~\eqref{eq:feas}}
\label{app:apps}

We spell out each application of Section~\ref{sec:intro}, give its
underlying optimisation problem, and identify the structural set $\Kset$
and the spectral set $\spec$ so that it matches the template
\[
\text{find } X \in \Kset \cap \Cset,
\qquad
\Cset = \{X : \sgm(X) \in \spec\},
\]
optionally with the proximity objective
$\min_X \tfrac12\|X-A\|_F^2$ that selects the nearest admissible matrix
to a datum $A$. Throughout, $\spec$ is taken on the ordered nonnegative
cone $\R^r_{\ge0,\downarrow}$.

\subsection{Preconditioning to a target condition number}
\emph{Underlying problem.} Given $A$ and (optionally) a prescribed
sparsity / affine pattern, find the nearest matrix whose condition
number is at most $\kappa_{\max}$:
\[
\min_{X}\ \tfrac12\|X-A\|_F^2
\quad\text{s.t.}\quad
X \in \Kset_{\mathrm{struct}},\ \
\sgm_1(X) \le \kappa_{\max}\,\sgm_r(X).
\]

\emph{Identification.} Take $\Kset = \Kset_{\mathrm{struct}}$ (an affine
or box set, closed convex) and
$\spec = \{ s : s_1 \le \kappa_{\max}\, s_r \}$, a single linear
inequality, hence convex. Then $\Cset$ is exactly the condition-number
set and the task is~\eqref{eq:feas} with a proximity objective. The
constraint lower-bounds $\sgm_r$, so $\Cset$ is non-convex.

\subsection{Weighted nuclear-norm low-rank recovery}
\emph{Underlying problem.} With a linear measurement operator
$\mathcal{A}(X)=b$ and an entrywise bound $\beta$, seek a matrix
consistent with the data whose weighted nuclear norm is capped:
\[
\text{find } X
\quad\text{s.t.}\quad
\mathcal{A}(X)=b,\ \ |X_{ij}|\le\beta,\ \
\sum_{i} w_i\,\sgm_i(X) \le \tau,
\]
with weights $w \in \R^r_{>0}$.

\emph{Identification.}
$\Kset = \{X : \mathcal{A}(X)=b,\ |X_{ij}|\le\beta\}$ (affine $\cap$ box,
closed convex), and $\spec=\{s : w^\top s \le \tau\}$ (a half-space on
the ordered cone, convex for every $w$). The projection $\proj_\spec$ is
the small QP of Section~\ref{sec:polyhedral}. The matrix set $\Cset$ is
convex iff the weights are non-increasing; the WNNM
weighting of~\cite{gu2017wnnm} (small weights on large singular values)
is the non-convex regime, handled here as a heuristic.

\subsection{Gain shaping for system matrices}
\emph{Underlying problem.} For a discrete-time linear map
$x_{k+1}=Xx_k$ with a prescribed sparsity pattern, find the nearest
admissible $X$ whose gains lie in a band:
\[
\min_X \tfrac12\|X-A\|_F^2
\quad\text{s.t.}\quad
X \in \Kset_{\mathrm{patt}},\ \
\ell \le \sgm_i(X) \le u\ \ \forall i.
\]

\emph{Identification.} $\Kset=\Kset_{\mathrm{patt}}$ (the pattern is an
affine constraint, closed convex) and $\spec=\{s : \ell \le s_i \le u\}$
(a box on the ordered cone, convex). Choosing $u<1$ certifies a
contraction since $\rho(X)\le\sgm_1(X)$. The lower bound $\ell$ makes
$\Cset$ non-convex.

\subsection{Spectrum shaping and whitening}
\emph{Underlying problem.} Find a matrix near $A$ with a flattened
spectrum, either by bounding the spread or the successive gaps:
\[
\min_X \tfrac12\|X-A\|_F^2
\quad\text{s.t.}\quad
X\in\Kset,\ \
\sgm_1(X)-\sgm_r(X)\le\delta,
\]
or with the successive-gap constraint
$\sgm_i(X)-\sgm_{i+1}(X)\le\delta$ in place of the spread.

\emph{Identification.} $\Kset$ is the proximity / structural set (closed
convex) and $\spec=\{s : s_1-s_r\le\delta\}$ or
$\{s : s_i-s_{i+1}\le\delta\}$ (linear inequalities on the ordered cone,
convex). The limit $\delta\to0$ forces all singular values equal, i.e.\
a scaled orthogonal (whitening) matrix. Both constraints bound small
singular values from below relative to the large ones, so $\Cset$ is
non-convex.

\section{Douglas--Rachford and ADMM for~\eqref{eq:feas}}
\label{app:dradmm}

Both methods reuse the \emph{same} two projection oracles as AP --- the
structural projector $\proj_\Kset$ and the spectral projector
$\proj_\Cset(A)=U\diag(\proj_\spec(\sgm(A)))V^\top$ --- and differ only
in how they combine them. We state each specialised to the two-set
feasibility problem $\Kset\cap\Cset$.

\subsection{Douglas--Rachford (DR)}
Let the reflections be $R_\Kset = 2\proj_\Kset - I$ and
$R_\Cset = 2\proj_\Cset - I$. DR iterates the averaged reflection
operator $T=\tfrac12\bigl(I+R_\Kset R_\Cset\bigr)$ on an auxiliary
variable $Z$. Unrolling $T$ gives the projector-only form
\begin{align}
Y_k     &= \proj_\Cset(Z_k), \notag\\
X_k     &= \proj_\Kset(2Y_k - Z_k), \label{eq:DR}\\
Z_{k+1} &= Z_k + (X_k - Y_k). \notag
\end{align}
The auxiliary $Z_k$ need not lie in either set; the \emph{solution
estimate} is the shadow $Y_k=\proj_\Cset(Z_k)$, which converges to a
point of $\Kset\cap\Cset$ when the intersection is nonempty (for convex
$\Kset,\Cset$). Relative to AP~\eqref{eq:AP}, DR replaces the plain
composition $\proj_\Kset\circ\proj_\Cset$ by the reflect--reflect--average
step, which lets it escape the corner stalling that AP can exhibit on
polyhedral $\Kset$.

\subsection{ADMM}
Recast~\eqref{eq:feas} as the consensus problem
\[
\min_{X,Y}\ \iota_\Cset(X) + \iota_\Kset(Y)
\quad \text{s.t.}\quad X=Y,
\]
where $\iota$ is the $\{0,\infty\}$ indicator. With scaled dual variable
$U$ and penalty $\rho>0$, the proximal map of an indicator is the
Euclidean projection, so the ADMM updates reduce to
\begin{align}
X_{k+1} &= \proj_\Cset(Y_k - U_k), \notag\\
Y_{k+1} &= \proj_\Kset(X_{k+1} + U_k), \label{eq:ADMM}\\
U_{k+1} &= U_k + X_{k+1} - Y_{k+1}. \notag
\end{align}
The penalty $\rho$ cancels --- the proximal maps of indicators are
projections regardless of $\rho$ --- so feasibility ADMM is
parameter-free here. For two sets this iteration is precisely DR applied
to the dual problem and produces the same shadow sequence; both reduce to
AP only when the dual offset $U$ (equivalently $Z_k-Y_k$ in DR) vanishes.

In every case the only spectral computation is the single SVD inside
$\proj_\Cset$. We report plain AP in the body and include
\eqref{eq:DR}--\eqref{eq:ADMM} for completeness.

\bibliographystyle{IEEEtran}
\bibliography{references}

\end{document}